\documentclass[twoside,11pt]{amsart}

\usepackage[english]{babel}
\usepackage[utf8x]{inputenc}
\usepackage{url}
\usepackage{amsmath}
\usepackage{graphicx}
\usepackage[colorinlistoftodos]{todonotes}
\usepackage{amsmath,amscd,amsthm}
\usepackage{amstext, amsfonts, a4}
\usepackage{amssymb}
\usepackage{tikz-cd}
\usepackage{fancyhdr}
\usepackage{enumerate}
\usepackage{cleveref}
\usepackage{xcolor}

\numberwithin{equation}{section}
\newtheorem{thm}[equation]{Theorem}

\newtheorem{prop}[equation]{Proposition}

\newtheorem{lem}[equation]{Lemma}

\theoremstyle{definition}

\newtheorem{ex}[equation]{Example}

\newtheorem{pb}[equation]{Problem}

\renewcommand{\dim}{\operatorname{\mathsf{dim}}}
\renewcommand{\deg}{\operatorname{\mathsf{deg}}}

\newcommand\ind{\operatorname{\mathsf{ind}}}

\newcommand\End{\operatorname{\mathsf{End}}}
\newcommand\Br{\operatorname{\mathsf{Br}}}

\newcommand\id{\operatorname{\mathsf{id}}}

\newcommand\hh{\operatorname{\mathbb{H}}}
\newcommand\N{\operatorname{N}}

\newcommand{\can}{\operatorname{\mathsf{can}}}

\newcommand{\vf}{\varphi}
\newcommand{\mg}[1]{{#1}^{\times}}
\newcommand{\sq}[1]{{#1}^{\times 2}}
\newcommand{\scg}[1]{\mg{#1}/\sq{#1}}
\newcommand{\s}{\sigma}
\newcommand{\nat}{\mathbb{N}}

\renewcommand{\mod}{\,\mathsf{mod}\,}

\newcommand{\la}{\langle}
\newcommand{\ra}{\rangle}
\newcommand{\lla}{\la\!\la}
\newcommand{\rra}{\ra\!\ra}      

\newcommand{\an}{\mathsf{an}}

\renewcommand{\leq}{\leqslant}

\newcommand\Ad{\operatorname{\mathsf{Ad}}}
\newcommand\ad{\operatorname{\mathsf{ad}}}

\newcommand{\I}{\mathsf{I}}

\renewcommand{\N}{\mathsf{N}}

\newcommand{\disc}{\mathsf{disc}}
\newcommand{\Sim}{{\bf\mathsf{Sim}}}
\newcommand{\PSim}{{\bf\mathsf{PSim}}}
\newcommand{\G}{\mathsf{G}}

\newcommand{\Hyp}{\mathsf{Hyp}}

\renewcommand{\setminus}{\smallsetminus}

\renewcommand{\deg}{\mathsf{deg}}
\renewcommand{\dim}{\mathsf{dim}}
\renewcommand{\leq}{\leqslant}

\newcommand{\wi}{\mathsf{i}}

\renewcommand{\setminus}{\smallsetminus}

\makeatletter
\newcommand{\bigperp}{%
  \mathop{\mathpalette\bigp@rp\relax}%
  \displaylimits
}

\newcommand{\bigp@rp}[2]{%
  \vcenter{
    \m@th\hbox{\scalebox{\ifx#1\displaystyle2.1\else1.5\fi}{$#1\perp$}}
  }%
}
\makeatother

\title[$R$-triviality for projective similitudes]{$R$-triviality for adjoint classical groups of type $C$}

\date{26 February, 2026}

\author{M.~Archita}

\address{University of Antwerp, Department of Mathematics, Antwerp, Belgium.}

\email{archita.mondal@uantwerpen.be}

\begin{document}

\maketitle

\begin{abstract}
    For a central simple algebra with a symplectic involution $(A,\s)$ over a field of characteristic different from $2$, we show that its group of projective similitudes ${\bf PSim}(A,\s)$ is $R$-trivial in two new cases.


\medskip\noindent
{\sc Classification (MSC 2020):} 11E04, 
11E57, 
11E81, 
14E08, 
20G15 
\end{abstract}

\section{Introduction}

Let $K$ be a field of characteristic different from $2$. 
By Weil's classification \cite{Wei61}, any absolutely simple adjoint algebraic group of type $C_n$ over $K$ is isomorphic
to the algebraic group of projective similitudes ${\bf PSim}(A,\s)$ for a central simple algebra $A$ of degree $2n$ over $K$ endowed with a symplectic involution $\s$. In \cite[\S 2]{Mer96}, it is shown that ${\bf PSim}(A,\s)$ is stably rational, and in particular $R$-trivial, when $n$ is odd.

In this article we  study the $R$-triviality of ${\bf PSim}(A,\s)$ when $n$ is even. In \cite[\S 2]{Mer96}, it is shown that the such groups are rational when $n= 2$. 
When $n=4$ and $\ind(A)\leq 2$ we will show that ${\bf PSim}(A,\s)$ is $R$-trivial (\Cref{4}). On the other hand, Berhuy, Monsurrò and Tignol \cite[Cor.~9]{BMT04} gave an example with $n=4$ and $\ind(A)=4$ where ${\bf PSim}(A,\s)$ is not $R$-trivial; see \Cref{8-odd multiple}.
When $n=6$, we establish $R$-triviality of ${\bf PSim}(A,\s)$ in the case where $\s$ has trivial discriminant (\Cref{6}).
\section{Preliminaries}
 For a finite field extension $L/K$, we write 
$\N_{L/K}:L\to K$ for the norm map, and we abbreviate $$\N_{L/K}^\ast=\N_{L/K}(\mg{L})\,.$$

For standard terminology and basic facts from quadratic forrm theory, we refer to \cite{Lam05}. Let $\vf = (V, B)$ be a quadratic form over $K$, where $V$ is a finite-dimensional $K$-vector space and $B$ a symmetric bilinear form $B:V\times V \rightarrow K$.   
The \emph{dimension of $\vf$} is the dimension of the vector space $V$ over $K$.
If $B$ is non-degenerate then $\vf$ is called {\it regular}. In this article, by a \emph{quadratic form} we mean a regular quadratic form. A quadratic form of dimension $n$ is equivalent to a diagonal form $a_1 X_1^2 + a_2 X_2^2 + \cdots + a_n X_n^2$, 
with $a_1,\dots,a_n \in \mg{K}$, which is denoted by $\langle a_1, \dots , a_n \rangle$. We denote by $\disc(\vf)$ the \emph{discriminant of $\vf$}, given as the class $(-1)^{\binom{n}{2}}a_1\dots a_n\sq{K}$ in the square-class group $\scg{K}$ for an arbitrary diagonalisation $\la a_1,\dots,a_n\ra$ of $\vf$.
We denote by $\vf_\an$ the anisotropic part and by $\wi(\vf)$ the Witt index of $\vf$.
They are determined by having that $\vf_\an$ is anisotropic and 
$\vf\simeq \vf_\an\perp \wi(\vf)\times \hh$, where $\hh$ denotes the hyperbolic plane over $K$.
We call $\vf$ \emph{isotropic} if $\wi(\vf)>0$, and \emph{anisotropic} otherwise.
We call $\vf$ \emph{hyperbolic} if $\vf_\an$ is trivial, i.e. if $\vf\simeq i\times \hh$ for some $i\in\nat$.  
We denote $\G(\vf)=\{a\in\mg{K}\mid a\vf\simeq \vf\}$, which is a subgroup of $\mg{K}$.
For $n\in\nat$ and $a_1,\dots,a_n\in\mg{K}$, we denote by $\lla a_1,\dots,a_n\rra$ the $n$-fold Pfister form $\la 1,-a_1\ra\otimes\dots\otimes\la 1,-a_n\ra$ over $K$.

Let $A$ be a $K$-algebra. 
We call $A$ a \emph{division algebra} if every nonzero element of $A$ is invertible. 
The center of $A$ is denoted by $Z(A)$. We call $A$ a \emph{central simple $K$-algebra} if $Z(A)=K$, $A$ is finite-dimensional over $K$ and the only two sided ideals of $A$ are $0$ and $A$. Note that, $\dim_{K}(A)=n^2,$ for some positive integer $n$, and we set $\deg(A)=n$. By Wedderburn's theorem, $A\simeq M_r(D)$, for some $r\in \nat$ and where $D$ is a finite-dimensional central division $K$-algebra (unique up to $K$-isomorphism). When $D=K$, we say $A$ is a \emph{split}. We call the degree of $D$ as a $K$-algebra the {\it index of} $A$ and it is denoted by $\ind(A)$. 

A \emph{$K$-involution} (or \emph{involution of the first kind}) on $A$ is a $K$-linear map $\sigma:A\rightarrow A$  such that $\s(xy)=\s(y)\s(x)$ for every $x,y\in A$ and $\s^2=\id_A$.
We call a pair $(A,\s)$ of a central simple $K$-algebra $A$ with a $K$-involution $\s$ a \emph{$K$-algebra with involution}.

Involutions of the first kind are distinguished into two \emph{types}, \emph{orthogonal} and \emph{symplectic}, cf. \cite[\S 2]{KMRT98}. Our focus here lies on symplectic involutions. They are characterized by the fact that they become adjoint to an alternating bilinear form after a field extension.

Consider a quadratic form $\vf=(V,B)$. Consider the split central simple $K$-algebra
$\End_{K} (V)$ and $B:V\times V\rightarrow K$ is a symmetric bilinear form. Let $\ad_B : \End_{K} (V) \rightarrow \End_{K} (V)$ denote the involution determined by the formula
$$B( f (u), v) = B(u, \ad_B( f )(v))$$ for all $u, v \in V$ and $f \in \End_{K} (V)$.
We call $\ad_B$ the \emph{adjoint involution of $\vf$}. Furthermore, we put $\Ad(\vf)=(\End_K (V), \ad_B)$ and call it the \emph{adjoint algebra with involution of $\vf$}.

Let $(A, \sigma)$ be an $K$-algebra with involution. An element $a\in \mg{A}$ is said to be a 
{\it similitude of $\s$} if $\sigma(a)a \in \mg{K}$. 
The similitudes of $(A, \sigma)$ form a group, which is denoted by 
$\Sim(A, \sigma)$.
In $\mg{K}$, we consider the subgroup
$$\G(A,\s)=\{\s(x)x\mid x\in\mg{A}\}\cap\mg{K}\,.$$
The elements of this group are called the \emph{multipliers of similitude of $(A,\s)$}.
We set $\PSim(A,\s)=\Sim(A,\s)/\mg{K}$ and call its elements \emph{projective similitudes of $(A,\s)$}.
We refer to \cite[\S12]{KMRT98} for a discussion of similitudes and their multipliers.

For a field extension $K'/K$ we denote by $(A_{K'}, \s_{K'})=(A\otimes_{K}K', \s\otimes\id_{K'})$ the algebra with involution $(A,\s)$ when considered over the field $K'$. 
Letting ${\bf PSim}(A,\s)(K')=\PSim(A_{K'},\s_{K'})$ for $K'/K$ defines an algebraic group ${\bf PSim} (A,\sigma)$ over $K$, now as a $K$-variety whose $K$-rational points are given by $\PSim(A,\s)$.

Let $(A,\s)$ be $K$-algebra with involution.
We call $\s$ \emph{hyperbolic} if there exists an element $e\in A$ such that $e^2=e$ and $\s(e)=1-e$.

We write $\Hyp(A,\s)$ for the subgroup of $\mg{K}$ generated by the sets $\N_{L/K}^\ast$ where $L/K$ ranges over all finite field extensions  such that $\s_L$ hyperbolic.

 The following theorem due to 
Merkurjev characterises the 
group of 
$R$-equiva\-lence classes of 
${\bf PSim} (A, \s)$ in terms of the multipliers 
of similitudes of $(A, \s)$. 

\begin{thm}[Merkurjev]\label{mer}
We have  $\sq{K}\cdot\Hyp(A,\s)\subseteq \G(A,\s)$ and 
$$ 
{\bf PSim} (A,\s)(K)/R \simeq  
\G(A,\s)/ \sq{K}\cdot\Hyp(A,\s) .  
$$ In particular, the group ${\bf PSim} (A,\s)$ is $R$-trivial if and only if,  for every field extension $K'/K$, one has $\G(A_{K'},\s_{K'})=\sq{K'}\cdot\Hyp(A_{K'},\s_{K'})$.
\end{thm} 
\begin{proof}
    See \cite[Theorem $1$]{Mer96}.
\end{proof}

\section{Main Results}

A \emph{$K$-quaternion algebra} is a central simple $K$-algebra of degree $2$. Consider a $K$-quaternion algebra $Q$. Denote the {\it canonical involution} on $Q$ by $\can_Q$. It is a $K$-involution of symplectic type. We denote the norm form of $Q$ by $\N_Q$ and recall that $\N_Q$ is a $4$-dimensional quadratic form over $K$.

\begin{prop}\label{initial}
Let $Q$ be a $K$-quaternion  algebra and $\vf$ a quadratic form over $K$. Let $\pi=\N_Q$.
Let  $(A,\s)= \Ad(\vf)\otimes (Q,\can_Q)$. Then $\G(A,\s)=\G(\vf\otimes \pi)$ and $\Hyp(A,\s)=\Hyp(\vf\otimes \pi)$. Furthermore, for any field extension $K'/K$, we have $${\bf PSim}(A,\s)(K')/R\simeq {\bf PSim}(\vf\otimes\pi)(K')/R\,.$$
\end{prop}
\begin{proof}
 By \cite[Theorem 4.6]{BU18}, for any quadratic form $\rho$ over $K$ we have that  $\Ad(\rho)\otimes (Q,\can_Q)$ is hyperbolic if and only if $\rho\otimes \pi$ hyperbolic. Since by \cite[Proposition 3.2]{BU18}, $\Ad(  \vf\otimes \pi )\simeq \Ad( \vf )\otimes \Ad( \pi )$ holds, we have, for any $\lambda\in \mg{K}$ that $\Ad(\lla \lambda\rra )\otimes (A,\s)\simeq \Ad(\lla \lambda \rra \otimes \vf)\otimes (Q,\can_Q)$, whereby $\Ad(\lla \lambda \rra)\otimes (A,\s)$ is hyperbolic if and only if $\lla \lambda \rra \otimes \vf\otimes \pi$ is hyperbolic. Thus, $\G(A,\s)=\G(\vf\otimes \pi)$.

     Note that, for any finite field extension $L/K$,  we have $(\vf\otimes \pi)_L$ is hyperbolic if and only if $\s_L$ is hyperbolic. Hence $\Hyp(A_L,\s_L)$ and $ \Hyp((\vf\otimes \pi)_L)$ are generated by the subgroups $\N_{L/K}^*$ for the same extensions $L/K$. Thus, we have $\Hyp(A,\s)=\Hyp(\vf\otimes \pi)$.

     Hence for any field extension $K'/K$,  we have $\G(A_{K'},\s_{K'})=\G((\vf\otimes \pi)_{K'})$ and $\Hyp(A_{K'},\s_{K'})=\Hyp((\vf\otimes \pi)_{K'})$. Thus by \Cref{mer}, we obtain that ${\bf PSim}(A,\s)(K')/R\simeq {\bf PSim}(\vf\otimes\pi)(K')/R\,.$
\end{proof}

\begin{thm}\label{4}
    Let $(A,\s)$ be $K$-algebra with symplectic involution of degree $8$. Assume that $\ind(A)\leq 2$. Then ${\bf PSim}(A,\s)$ is $R$-trivial.
\end{thm}

\begin{proof}
As $(A,\s)$ is an algebra with symplectic involution with $\ind(A)\leq 2$, by \cite[Prop. 3.4]{BU18} we have $(A,\s)\simeq \Ad(\vf)\otimes (Q,\can_Q)$ for some quadratic form $\vf$ over $K$ and a quaternion algebra $Q$ over $K$. Write $Q=(a,b)$ with $a,b\in\mg{K}$. 
Its norm form is $\lla a,b\rra$.
Let $a_1,\dots,a_4\in\mg{K}$ be such that $\vf\simeq \la a_1,a_2,a_3,a_4\ra$.
Let $\psi=\lla a, b \rra \otimes \vf$. Now observe that $\la a_1, a_2, a_3, a_4\ra$ is Witt equivalent to $\la a_1, a_2, a_3, a_4\ra\perp \la a_1a_2a_3,-a_1a_2a_3\ra$, which is equal to $a_1\lla -a_1a_3,-a_1a_2\rra\perp a_4\lla a_1a_2a_3a_4\rra$. It follows that $\psi$ is Witt equivalent to a sum of a scaled $4$-fold and a scaled $3$-fold Pfister form. By \cite[Theorem 3.5] {AB25} we obtain that ${\bf PSim}(\psi)$ is $R$-trivial. Thus by \Cref{initial}, we conclude that ${\bf PSim}(A,\s)$ is $R$-trivial.
\end{proof}

A construction from \cite{BMT04} leads to examples of $K$-algebras with symplectic involution $(A,\s)$ where $\deg A=8$, $\ind A=4$ and ${\bf PSim}(A,\s)$ is not $R$-trivial.

\begin{ex}\label{8-odd multiple}
   Let $K=K_0(a,b)$, the rational function field in two variables $a,b$ over a subfield $K_0$.
   Consider $(B,\tau)$ be a central simple $K_0$-algebra of degree $4$ and $\ind(B)=2$ with an orthogonal involution $\tau$ of nontrivial discriminant. Let $Q=(a,b)$ and let $\can_Q$ be its canonical involution.
  
   Set $A=B\otimes Q$ and $\s=\tau\otimes \can_Q$. 
   Then $\deg(A)=8$, $\ind(A)= 4$ and $\s$ is a symplectic involution on $A$ , by \cite[Prop.~2.23]{KMRT98}. 
   By \cite[Cor.~9]{BMT04}, in this case ${\bf PSim}(A,\s)(K)/R\neq \{1\}$.
\end{ex}

Next we consider symplectic involutions on algebras of degree $12$ with some conditions on cohomological invariants. Let $n\in\nat$. 
We denote by $\I^nK$ the $n$th power of the fundamental ideal $\I K$ of the Witt ring of $K$.
We write $\vf\in\I^n K$ if the Witt equivalence class of $\vf$ lies in $\I^n K$.
Note that $\vf\in\I^2K$ if and only if $\vf$ has even dimension and trivial discriminant. 

\begin{prop}\label{index}
Let $\pi$ be a Pfister form and $\psi$ be an even-dimensional quadratic form over $K$. Let $\vf=\pi\otimes \psi$. Then, for every field extension $L/K$, there exists a quadratic form $\theta$ over $L$ such that $\dim\theta\equiv \dim\psi  \mod 2$ and $(\vf_L)_{\an}=\pi_L\otimes \theta$.
\end{prop}
\begin{proof}
    See \cite[Prop.~3.4]{AB26}. 
\end{proof}

\begin{lem}\label{L:beta}
Let $\vf$ be a quadratic form over $K$ and $a\in\G(\vf)$.
Then  $\vf$ is hyperbolic or there exists a quadratic field extension $L/K$ with $\wi(\vf_L)>\wi(\vf)$ and $a\in\N_{L/K}^\ast$.
\end{lem}
\begin{proof}
    See \cite[Lemma~3.1]{AB25}.
\end{proof}

    \begin{lem}\label{24}
      Let $\pi$ a $2$-fold Pfister form. Let $\psi$ be a $6$-dimension quadratic form and $\vf=\pi\otimes \psi$ in $\I^4K$ over $K$. 
      Then, for every $c\in \G(\vf)$, there exists a field extension $M/K$ which is trivial, quadratic or biquadratic such that $\vf_{M}$ is hyperbolic and $c\in\sq{K}\cdot \N_{M/K}^\ast$. In particular, $\G(\vf)=\sq{K}\cdot\Hyp(\vf)$.
    \end{lem}
    \begin{proof}
    Consider $c\in \G(\vf)$. 
    Then $\lla c \rra \otimes \vf $ is isotropic. 
    If $\vf$ is hyperbolic, then choose $M=K$.
    
    Assume now that $\vf$ is not hyperbolic.
Then, by \Cref{L:beta}, there exists a quadratic field extension $L/K$ such that  $\wi(\vf_L)>\wi(\vf)$ and $c\in\N_{L/K}^\ast$.
    Let $d\in\mg{K}\setminus\sq{K}$ be such that $L=K(\sqrt{d})$.
    If $\vf_L$ is hyperbolic, then take $M=L$. Assume now that $\vf_L$ is not hyperbolic.
    By \cite[Cor.~22.12]{EKM08}, we can write $\vf_\an\simeq \lla d\rra\otimes \vartheta\perp\psi$ for two quadratic forms $\vartheta$ and $\psi$ over $K$ such that $\psi_L$ is anisotropic. Then $(\vf_L)_\an=\psi_L$. Now by \Cref{index}, there exists a quadratic form $\eta$ over $L$  with $\dim(\eta)\leq 4$ such that $\psi_L=(\vf_L)_{\an}=\pi\otimes \eta$.
    We have that $c\in\G(\vf)\cap\G(\lla d\rra)\subseteq\G(\psi)$ and $\dim(\psi)\leq 16$. By \Cref{L:beta}, there exists a quadratic extension $L'/K$ such that $\psi_{L'}$ is isotropic and $c\in\N_{L'/K}^\ast$. Since $\psi_{L}$ is anisotropic, it follows that $L/K$ and $L'/K$ are linearly disjoint.
    We set $M=L\otimes_KL'$.
    By \cite[Lemma~3.3]{AB25}, we have  $\N_{L/K}^\ast\cap\N_{L'/K}^\ast=\sq{K}\cdot\N_{M/K}^\ast$.
    Since $\psi_{L'}$ is isotropic, so is $\psi_M$. 
    Since $\vf_M$ is Witt equivalent to $\psi_M$, we have $\dim((\vf_M)_\an)<\dim(\psi)=16$.
    By our assumption, $\vf_M\in \I^4M$. Hence, by the Arason-Pfister Hauptsatz \cite[Theorem 23.7]{EKM08} we conclude that $\vf_M$ is hyperbolic. Therefore we have  $c\in\sq{K}\cdot\N_{M/K}^\ast\subseteq \sq{K}\cdot\Hyp(\vf)$.
\end{proof}
The cohomology sets $H^1(K, \mu_2)$ and $H^2(K, \mu_2)$ are identified with the quotient $\mg{K}/\sq{K}$ and with the $2$-torsion subgroup of the
Brauer group $\Br_2(K)$ respectively. Let $(A,\s)$ be a central simple algebra with symplectic involution over $K$. Then $[A]\in \Br_2(K)$. When $2\ind(A)$ divides $\deg(A)$, the discriminant of $(A,\s)$ is defined in \cite{BMT03} as a class in the cohomology set $H^3(K,\mu_2),$ and is denoted by $\Delta(\s)$. 
\begin{thm}\label{6}
    Let $(A,\s)$ be a central simple algebra of $\deg(A)=12$ with symplectic involution. Assume that $\ind(A)\leq 2$ and trivial discriminant. Then  ${\bf PSim}(A,\s)$ is $R$-trivial.
\end{thm}
\begin{proof}
     As $(A,\s)$ be a central simple algebra with symplectic involution with $\ind(A)\leq 2$, by \cite[Prop. 3.4]{BU18} we have $(A,\s)\simeq \Ad(\vf)\otimes (Q,\can_Q)$ for some $6$-dimensional quadratic form $\vf$ over $K$ and $Q$ a quaternion algebra over $K$. 
     Write $Q=(a,b)$, with $a,b\in\mg{K}$. Then its norm form is $\lla a,b\rra$.
 
Let $\psi:=\lla a, b \rra \otimes \vf$.
Denote $\disc(\vf)=c\sq{K}$. 
Then $\vf \equiv \lla c \rra \mod \I^2K$ and hence $\psi=\lla a, b\rra \otimes \vf \equiv \lla a, b ,c\rra \mod \I^4K$. 
By \cite[Main Theorem]{ST04}, we have $\Delta(\s)=[Q]\cup(\disc(\vf))$ in $H^3(K,\mu_2)$.
By our hypothesis, $\Delta(\s)=0$.
Hence 
$(a)\cup(b)\cup (c)=(\disc(\vf))\cup[Q]=0$.   
Hence, by \cite[Theorem 5.1]{MS91} together with \cite[Theorem 4.1]{Mil70}, we have that $\lla a,b,c\rra$ is hyperbolic.

Hence, $\psi\in \I^4K$. 
By \Cref{24}, it follows that $\G(\psi)=\sq{K}\cdot\Hyp(\psi)$. 
By \Cref{initial}, we have $\G(A,\s)=\G(\psi)$ and $\Hyp(A,\s)=\Hyp(\psi)$. 

 For any field extension $K'/K$, if $c\in \G(\psi_{K'})$ by \Cref{24}, then there exists a field extension $M/K'$ which is trivial, quadratic or biquadratic such that $\psi_M$ is hyperbolic and $c\in\sq{K'}\cdot \N_{M/K'}^\ast$. In particular, $\G(\psi_{K'})=\sq{K'}\cdot\Hyp(\psi_{K'})$ and hence ${\bf PSim}(\psi)$ is $R$-trivial. Thus ${\bf PSim}(A,\s)$ is $R$-trivial.
\end{proof}

\begin{pb}
    Let $(A,\s)$ be a central simple algebra of $\deg(A)=12$ with symplectic involution. Assume that $\ind(A) = 4$. Then either show that ${\bf PSim}(A,\s)$ is $R$-trivial or construct an example to describe ${\bf PSim}(A,\s)$ is not $R$-trivial.
\end{pb}

\subsection*{Acknowledgments}
This work was supported by the \emph{Bijzonder Onderzoeksfonds, University of Antwerp} (project \emph{BOF Opvang MSCA IF}, ID 51418).

\section*{Declarations}

\subsection*{Data availability}
There is no associated data, which is not already contained in the text.

\subsection*{Conflict of interest} The author declares that there is no conflict of interest.

\bibliographystyle{plain}

\end{document}